\documentclass[11pt, a4paper]{article}

\usepackage[utf8]{inputenc}
\usepackage{amsmath, amsthm, amsfonts, amssymb}

\usepackage{authblk}

\providecommand{\keywords}[1]
{
  \small	
  \textbf{\textit{Keywords---}} #1
}

\newtheorem{lemma}{Lemma}[section]
\newtheorem{theorem}[lemma]{Theorem}
\newtheorem{proposition}[lemma]{Proposition}
\newtheorem{corollary}[lemma]{Corollary}

\newtheorem{result}{Theorem}

\newtheorem{corollaryresult}[result]{Corollary}

\theoremstyle{definition}

\newtheorem{example}[lemma]{Example}

\newcommand{\abs}[1]{\ensuremath{\left| #1 \right|}}
\newcommand{\op}{\operatorname}
\newcommand{\ce}[2]{\operatorname{C}_{#1}(#2)}
\newcommand{\no}[2]{\operatorname{N}_{#1}(#2)}
\newcommand{\ze}[1]{\operatorname{Z}(#1)}

\newcommand{\rad}[2]{\op{O}_{#1}(#2)}

\newcommand{\syl}[2]{\op{Syl}_{#1}\left(#2\right)}
\newcommand{\hall}[2]{\op{Hall}_{#1}\left(#2\right)}

\newcommand{\irr}{\operatorname{Irr}}
\newcommand{\lin}{\operatorname{Lin}}
\newcommand{\B}[1]{\operatorname{B}_{#1}}
\newcommand{\X}[1]{\operatorname{X}_{#1}}

\newcommand{\cd}{\op{cd}}
\newcommand{\BBcd}[1]{\operatorname{cd}_{\operatorname{B}_{#1}}}%

\renewcommand{\phi}{\varphi}

\title{Character degrees in $\pi$-separable groups}

\date{\today}

\author{Nicola Grittini \footnote{Author's research is partially supported by INdAM - Istituto Nazionale di Alta Matematica F. Severi. The author also thanks professor Silvio Dolfi for his precious advices.}}
\affil{Università degli Studi di Firenze}

\begin{document}

\maketitle

\begin{abstract}
If a group $G$ is $\pi$-separable, where $\pi$ is a set of primes, the set of irreducible characters $\B{\pi}(G) \cup \B{\pi'}(G)$ can be defined. In this paper, we prove that there are variants of some classical theorems in character theory, namely the Theorem of Ito-Michler and Thompson theorem on character degrees, which involve irreducible characters in the set $\B{\pi}(G) \cup \B{\pi'}(G)$.

\end{abstract}

\section{Introduction}

The character theory of $\pi$-separable group was first introduced by Martin I. Isaacs in 1984 in a series of papers, the first of them being \cite{Is1}. In those papers, Isaacs introduced, for the $\pi$-separable group $G$, the family $\op{I}_{\pi}(G)$ of \textit{$\pi$-partial characters}, defined only on $\pi$-elements, and a family of irreducible lifts $\B{\pi}(G)$ for these characters. The aim of the paper was originally to generalize, for $\pi$-separable group, the concept of Brauer characters; in fact, if the group $G$ is $p$-solvable one has that $\op{I}_{p'}(G)=\op{IBr}_p(G)$.

In this paper, we study how the degrees of the $\B{\pi}$-characters (and, therefore, of the partial characters, too) influence the group structure.

%
%
%

In particular, we will focus on the degrees of the characters belonging to the set $\B{\pi}(G) \cup \B{\pi'}(G)$. This set is in general smaller then $\irr(G)$ and it is easy to find examples of a $\pi$-separable group $G$ and a character $\psi \in \irr(G)$ such that $\psi(1) \neq \chi(1)$ for any $\chi \in \B{\pi}(G) \cup \B{\pi'}(G)$. Nevertheless, in $\pi$-separable groups the degrees of the characters in $\B{\pi}(G) \cup \B{\pi'}(G)$ present some properties which are usually associated with the degrees of the characters in $\irr(G)$. In this paper, we present a version for the $\B{\pi}$-characters of two theorems about character degrees, which are sometimes considered as dual: the Theorem of Ito-Michler and the Thompson's Theorem on character degrees.

The famous Theorem of Ito-Michler says that a group $G$ has a normal abelian Sylow $p$-subgroup, for some prime $p$, if and only if $p$ does not divide the degree of any irreducible character of the group. If $G$ is $\pi$-separable, we see that there exists a version of the theorem involving only the characters in $\B{\pi}(G) \cup \B{\pi'}(G)$.

\begin{result}
\label{result:BpiItoMichler}
Let $G$ be a $\pi$-separable group and $p$ be any prime. Then $G$ has a normal abelian Sylow $p$-subgroup if and only if $p$ does not divide the degree of any character in $\B{\pi}(G) \cup \B{\pi'}(G)$.
\end{result}

Of course, there is an easy corollary following from this.

\begin{corollaryresult}
Let $G$ be a $\pi$-separable group and let $p$ be any prime. Then, $p$ divides the degree of some characters in $\irr(G)$ if and only if it divides the degree of some characters in $\B{\pi}(G) \cup \B{\pi'}(G)$.
\end{corollaryresult}

%
%
%
%

The well known Thompson Theorem on character degrees affirms that, if a prime $p$ divides the degree of every nonlinear irreducible character of a group $G$, then $G$ has a normal $p$-complement.

In \cite{NW}, Navarro and Wolf studied a variant of the theorem which involves more then one prime. Let $\irr_{\pi'}(G)$ be the set of irreducible characters which degree is not divided by any prime in $\pi$. To ask that $p$ divides the degree of every irreducible nonlinear character of $G$ is equivalent to ask that $\irr_{p'}(G)=\lin(G)$. In \cite[Corollary 3]{NW}, Navarro and Wolf considered more than one prime and proved that, if $G$ is a $\pi$-separable group and $H$ is a Hall $\pi$-subgroup, then $\irr_{\pi'}(G)=\lin(G)$ if and only if $G' \cap \no{G}{H}=H'$.

In this paper, a first result establishes an equivalence between the condition on character degrees studied in the aforementioned \cite[Corollary 3]{NW} and the same condition restricted to the set of characters $\B{\pi}(G) \cup \B{\pi'}(G)$.

\begin{result}
\label{result:thompson1}
Let $G$ be a $\pi$-separable group. Then, $\irr_{\pi'}(G)=\lin(G)$ if and only if $\irr_{\pi'}(G) \cap \big(\B{\pi}(G) \cup \B{\pi'}(G) \big) \subseteq \lin(G)$.
\end{result}

%
%
%

Afterwards, the paper focuses on variants of Thompson's Theorem considering only $\B{\pi}$-characters or $\B{\pi'}$-characters.

\begin{result}
\label{result:thompsonpipiprime}
Let $G$ be a $\pi$-separable group, let $H$ be a Hall $\pi$-subgroup for $G$ and let $N=\no{G}{H}$. Then,
\begin{itemize}
\item[a)] $\irr_{\pi'}(G) \cap \B{\pi}(G) \subseteq \lin(G)$ if and only if $G' \cap H = H'$;
\item[b)] $\irr_{\pi'}(G) \cap \B{\pi'}(G) \subseteq \lin(G)$ if and only if $G' \cap N \leq H$.
\end{itemize}

\end{result}

%
%

The reader may notice that this last result are actually strongly related with \cite[Corollary 3]{NW}. In fact, since all the proofs presented here are independent from \cite[Corollary 3]{NW}, Results \ref{result:thompson1} and \ref{result:thompsonpipiprime} can provide an alternative, even if not shorter, proof of it.

\section{Review of the $\pi$-theory}

In this section, we are going to recall briefly some essential concepts of the character theory of $\pi$-separable group. In \cite{Is1} the reader can find a more extensive one.

At first, it is needed to define another subset of the irreducible characters, the \textit{$\pi$-special characters}.

\begin{def}\emph{\cite{G}}
Let $\chi \in \irr(G)$, then $\chi$ it is said to be a \textit{$\pi$-special character} if, for any $M \leq G$ subnormal in $G$, every irreducible constituent of $\chi_M$ has order and degree which are both $\pi$-numbers.
\end{def}

We call $\X{\pi}(G)$ the set of all the $\pi$-special character of the group $G$.

The concept of $\pi$-special character is essential for further developments of the theory and for the definition of the set $\B{\pi}$-characters.

\begin{theorem}[\emph{\cite{Is1}}]
\label{theorem:nucleus}
Let $G$ be a $\pi$-separable group and let $\chi \in \irr(G)$. Then there exists a subgroup $W \leq G$, canonically defined up to conjugation, $\alpha \in \X{\pi}(W)$ and $\beta \in \X{\pi'}(W)$ such that $\chi=(\alpha\beta)^G$. If $\beta=1_W$, then the character $\chi$ is a $\B{\pi}$-character.
\end{theorem}

The relation between $\B{\pi}$ and $\pi$-special character is actually even stronger.

\begin{theorem}[\emph{\cite[Lemma 5.4]{Is1}}]
\label{prop:bpiespec}
Let $\chi \in \B{\pi}(G)$, then $\chi$ is $\pi$-special if and only if $\chi(1)$ is a $\pi$-number.
\end{theorem}

Moreover, if $\chi \in \irr_{\pi'}(G)$ there exists a stronger version of Theorem~\ref{theorem:nucleus}, due to Isaacs and Navarro.

\begin{theorem}[\emph{\cite[Theorem 3.6]{IN}}]
\label{theorem:maximalnucleus}
Let $G$ be a $\pi$-separable group, $H \in \hall{\pi}{G}$ and let $\chi \in \irr_{\pi'}(G)$. Then there exists a subgroup $H \leq W \leq G$, $\alpha \in \X{\pi}(W)$ linear and $\beta \in \X{\pi'}(W)$ such that $\chi=(\alpha\beta)^G$. Moreover, $W$ can be chosen as the (unique) maximal subgroup of $G$ such that $\alpha_H$ extends to $W$.
\end{theorem}

Let us recall the behaviour of the $\B{\pi}$-characters in relation with normal subgroups.

\begin{theorem}
\label{theorem:Bpinormalsubgroups}
Let $G$ be $\pi$-separable, let $M \unlhd G$ and let $\chi \in \B{\pi}(G)$, then every irreducible constituent of $\chi_M$ belongs ot $\B{\pi}(M)$.

On the other hand, let $\psi \in \B{\pi}(M)$; if $\abs{G:M}$ is a $\pi$-number, then every irreducible constituent of $\psi^G$ is in $\B{\pi}(G)$ while, if $\abs{G:M}$ is a $\pi'$-number, then there exists a unique irreducible constituent of $\psi^G$ which belongs to $\B{\pi}(G)$.

In particular, if $\psi \in \B{\pi}(M)$, then there exists always at least one character $\chi \in \irr(G \mid \psi)$ which belongs to $\B{\pi}(G)$.
\end{theorem}

\begin{proof}
It is a direct consequence of \cite[Theorem 6.2]{Is1} and \cite[Theorem 7.1]{Is1}
\end{proof}

A basilar property of the $\B{\pi}$-characters concerns their restriction to Hall $\pi$-subgroups.

\begin{theorem}[\emph{\cite[Theorem 8.1]{Is1}}]
\label{theorem:fongchar}
Let $\chi \in \B{\pi}(G)$ and let $H \in \hall{\pi}{G}$, then there exists an irreducible constituent $\phi$ of $\chi_H$ such that $\phi(1)=\chi(1)_{\pi}$. Moreover, for any irreducible constituent $\phi$ of $\chi_H$ such that $\phi(1)=\chi(1)_{\pi}$, the multiplicity of $\phi$ in $\chi_H$ is 1 and $\phi$ does not appear as an irreducible constituent of the restriction to $H$ of any other character in $\B{\pi}(G)$.
\end{theorem}

A character $\phi \in \irr(H)$ like the ones in Theorem~\ref{theorem:fongchar} is called \emph{Fong character} associated with $\chi$. It is in general quite difficult to tell if an irreducible character of the Hall $\pi$-subgroup $H$ is a Fong character associated with some $\B{\pi}$-character of $G$; a characterization is presented in \cite{Is2}. The problem is simpler, however, if one consider only the primitive characters of $H$.

\begin{theorem}[\emph{(\cite[Corollary 6.1]{Is2} or \cite[Theorem 5.13]{Is4})}]
\label{theorem:primitivefongchar}
Let $H$ be a Hall $\pi$-subgroup of a $\pi$-separable group $G$ and let $\phi \in \irr(H)$. If $\phi$ is primitive, then it is a Fong character associated with some character in $\B{\pi}(G)$. If $\phi_1$ is another primitive irreducible character of $H$, then $\phi$ and $\phi_1$ are associated with the same character in $\B{\pi}(G)$ if and only if they are $\no{G}{H}$-conjugated.
\end{theorem}

\section{Some examples}

Considering the nature of the results presented in this paper, a natural question a reader may ask is whether the set $\B{\pi}(G) \cup \B{\pi}(G)$ is actually strictly smaller then $\irr(G)$. This happens quite often. In fact, one of the properties of $\B{\pi}$-characters (see \cite[Theorem~9.3]{Is1}) is that $\abs{\B{\pi}(G)}$ is equal to the number of conjugacy classes of $\pi$-elements of $G$. Therefore, $\B{\pi}(G) \cup \B{\pi}(G) = \irr(G)$ if and only if each element of the $\pi$-separable group $G$ is either a $\pi$-element or a $\pi'$-element. This is proved in \cite[Lemma~4.2]{GR} to happen if and only if $G$ is a Frobenius or a 2-Frobenius group and each Frobenius complement and Frobenius kernel is either a $\pi$-group or a $\pi'$-group.

Let us call $\cd(G)$ the set of irreducible character degrees of $G$ and let us refer as $\BBcd{\pi}(G)$ and $\BBcd{\pi'}(G)$ to the sets of character degrees of, respectively, $\B{\pi}$-characters and $\B{\pi'}$-characters. Even when $\B{\pi}(G) \cup \B{\pi}(G)$ is strictly smaller then $\irr(G)$, it may happen that $\cd(G) = \BBcd{\pi}(G) \cup \BBcd{\pi'}(G)$. This happens, for example, if we consider the group $\op{SL}(2,3) \ltimes (\mathbb{Z}_3)^2$, with $\pi=\{2\}$, or the group $(\op{C}_3 \ltimes \op{C}_7) \wr \op{C}_2$, with $\pi=\{7\}$.

%
%
%

However, for a $\pi$-separable group $G$, in general we have that $\cd(G) \neq \BBcd{\pi}(G) \cup \BBcd{\pi'}(G)$. A first, obvious example of this fact is when $G=H \times K$, with $H$ a $\pi$-group and $K$ a $\pi'$-group, both nonabelian. In this case, in fact, we have that $\B{\pi}(G)=\irr(H)$ and $\B{\pi'}(G)=\irr(K)$.

Let us see some less trivial examples.

\begin{example}
\label{example:2}
A first example is derived directly form the trivial one. Let $G=H \times K$, with $H$ a $\pi$-group and $K$ a $\pi'$-group, both nonabelian, and let $\Gamma = G \; \wr \; \op{C}_2$. Suppose $2 \in \pi$. Let $\theta$ be a nonlinear character in $\irr(H)$ and let $\eta$ be a nonlinear character in $\irr(K)$. The character $(\theta \times 1_K) \times (1_H \times \eta)$ is an irreducible character of the base group $G \times G$ and it is not $\Gamma$-invariant, therefore it induces irreducibly to a character $\chi \in \irr(\Gamma)$.

We have that the $\pi$-part of the degree of $\chi$ is $\chi(1)_{\pi}=2\theta(1)>2$ and the $\pi'$-part of the degree is $\chi(1)_{\pi'}=\eta(1)>1$. Suppose there exists $\psi \in \B{\pi}(\Gamma) \cup \B{\pi'}(\Gamma)$ such that $\chi(1)=\psi(1)$ and let $\lambda_1 \times \lambda_2$ be a constituent of $\psi_{G \times G}$. Since $2 \in \pi$, $\lambda_1(1)_{\pi'} \lambda_2(1)_{\pi'}=\psi(1)_{\pi'} > 1$, thus, $\lambda_1,\lambda_2 \in \B{\pi'}(G)$. As a consequence, $\lambda_1(1)_{\pi} \lambda_2(1)_{\pi}=1$ while $\psi(1)_{\pi}>2$, a contradiction.

It follows that $\BBcd{\pi}(\Gamma) \cup \BBcd{\pi'}(\Gamma)$ is strictly smaller then $\cd(\Gamma)$.
\end{example}

The group in Example~\ref{example:2} still involves a group of type $H \times K$, with $H$ a $\pi$-group and $K$ a $\pi'$-group. However, this is not a necessary condition, as the next example shows.

\begin{example}
Let $G = H \ltimes M$, with $H=\op{SL}(2,3)$ acting canonically on the vector space $M=(\mathbb{Z}_3)^2$. Computing the character table of $G$, we can see that $\cd(G)=\{1,2,3,8\}$ and, with a little more work, it is not hard to prove that $\BBcd{3}(G)=\{1,8\}$ and $\BBcd{2}(G)=\{1,2,3\}$.

Now, let $\Gamma=G \; \wr \; \op{C}_2$, let $\theta \in \B{3}(G)$ of degree 8 and let $\eta \in \B{2}(G)$ of degree 3. The character $\theta \times \eta \in \irr(G \times G)$ induces irreducibly to $\Gamma$ and $\chi(1)=48$.

Suppose there exists $\psi \in \B{2}(\Gamma) \cup \B{3}(\Gamma)$ such that $\psi(1)=\chi(1)$ and let $\lambda_1 \times \lambda_2$ be an irreducible constituent of $\psi_{G \times G}$. Then, $\lambda_1$ and $\lambda_2$ are either both in $\B{2}(G)$ or they are both in $\B{3}(G)$. Moreover, since $\psi(1)=48$, then $\lambda_1(1)\lambda_2(1) \in \{24,48\}$. However, neither 24 nor 48 can be written as a product of two numbers in $\BBcd{2}(G)$ or as a product of two numbers in $\BBcd{3}(G)$. It follows that $48 \in \cd(\Gamma)$ but $48 \notin \BBcd{2}(\Gamma) \cup \BBcd{3}(\Gamma)$.
\end{example}

\section{Character degrees and normal subgroups}

In this section we are going to see the proofs of Theorem~\ref{result:BpiItoMichler}. We also give a different proof of a result which appears in \cite{Is4}.

The technique used to prove the results in this section mirrors the one used in \cite{DPSS}. In particular, the key result for the section is the following lemma, borrowed from \cite{DPSS}.

\begin{lemma}
\label{lemma:vanishonA}
Let $G$ be a group, let $N$ be a normal minimal $\pi'$-subgroup and let $M \unlhd G$ such that $M/N$ is an abelian $\pi$-group. Furthermore, suppose that $\rad{\pi}{M}=1$. Then, there exists a character $\chi \in \B{\pi'}(G)$ such that $\chi(1)$ is divided by $\abs{M:N}$.
\end{lemma}

\begin{proof}
Let $A$ be a complement for $N$ in  $M$, which exists by the Schur–Zassenhaus theorem. Since $\ce{A}{N} \leq \ze{M}$ and $A$ is a $\pi$-group, $\ce{A}{N} \leq \rad{\pi}{M} \leq \rad{\pi}{G}=1$; thus, $A$ acts faithfully on $N$. By \cite[Lemma 2.8]{DPSS}, there exists some character $\tau \in \irr(N)$ such that $\eta=\tau^M \in \irr(M)$. In particular, $\abs{M:N}$ divides $\eta(1)$. Since $\tau \in \B{\pi'}(N) = \irr(N)$ and $N$ is normal in $M$, by Theorem~\ref{theorem:Bpinormalsubgroups} the character $\eta$ is in $\B{\pi'}(M)$, too. It follows that $\B{\pi'}(G \mid \eta)$ is nonempty and $\abs{M:N}$ divides the degree of every character in $\B{\pi'}(G \mid \eta)$.
\end{proof}

Now, as anticipated, we are going to present a different proof of \cite[Theorem 3.17]{Is4} using Lemma~\ref{lemma:vanishonA}.

\begin{theorem}[\emph{\cite[Theorem 3.17]{Is4}}]
Let $G$ be $\pi$-separable, then $\B{\pi}(G)=\X{\pi}(G)$ if and only if $G$ has a normal $\pi$-complement.
\end{theorem}

\begin{proof}
Note at first that, if $G$ has a normal $\pi$-complement $H$, then it follows that $\B{\pi}(G)=\X{\pi}(G)=\irr(G/H)$. Thus, there is only one implication to be proved.

Let us assume that $\B{\pi}(G)=\X{\pi}(G)$ and prove the thesis by induction on $\abs{G}$. At first, let us assume that $\rad{\pi'}{G}=1$ since, otherwise, the thesis would follow by induction.

Let $N$ be a normal minimal subgroup of $G$ and suppose it to be a $\pi$-group. Since the hypothesis are preserved by factor groups, if $H$ is a Hall $\pi'$-subgroup of $G$, then by induction $HN$ is normal in $G$. In particular, it follows that there exists $K \lhd G$ such that $K/N$ is a $\pi'$-chief factor of $G$. Since $\abs{N}$ and $\abs{K/N}$ are coprime, at least one of them is odd and, thus, since an odd group is solvable and both $N$ and $K/N$ are normal minimal in $G$, at least one of them is abelian.

Suppose $K/N$ is abelian. Since we have assumed $\rad{\pi'}{G} = 1$, it follows by Lemma~\ref{lemma:vanishonA} that there exists a character in $\B{\pi}(G)$ which degree is divided by the $\pi'$-number $\abs{K:N}$, contradicting the hypothesis.

Suppose now that $K/N$ is not abelian, so $N$ has to be, and let $\lambda \in \irr(N)=\B{\pi}(N)$. If $\lambda$ is not $K$-invariant, then the degree of some $\theta \in \B{\pi}(K \mid \lambda)$ is divided by some primes in $\pi'$ and, therefore, so is the degree of some character $\chi \in \B{\pi}(G \mid \theta)$, in contradiction with the hypothesis. It follows that $K$ fixes every character of $N$ and, thus, it also centralizes $N$, since $N$ is abelian. If $B$ be a complement of $N$ in $K$, then it is normal in $K$. In particular, $1 < B=\rad{\pi'}{K} \leq \rad{\pi'}{G}$.

Therefore, we have that $\rad{\pi'}{G} \neq 1$ and the thesis follows by induction. 
\end{proof}

We now prove Theorem~\ref{result:BpiItoMichler}. We mention that the proof was simplified after some suggestions from an anonymous reviewer, who we thank.

\begin{proof}[Proof of Theorem~\ref{result:BpiItoMichler}]
It can be observed that there is little to prove in one direction, being it a consequence of the Ito-Michler theorem. Thus, we assume that $p$ does not divide the degree of any character in $\B{\pi}(G) \cup \B{\pi'}(G)$ and we first prove that there exists a normal Sylow $p$-subgroup. We argue by induction on $\abs{G}$.

Let $N$ be a minimal normal subgroup of $G$. Without loss of generality, we can assume $N$ to be a $\pi$-group. Suppose $p \in \pi$. By induction, let $K/N$ be a normal Sylow $p$-subgroup of $G/N$, then $K$ is a normal $\pi$-subgroup of $G$ which contains a Sylow $p$-subgroup $P \in \syl{p}{G}$. If $P$ is normal abelian in $K$, then it is also in $G$. Otherwise, there exists $\theta \in \irr(K)=\B{\pi}(K)$ such that $p \mid \theta(1)$ and, by Theorem~\ref{theorem:Bpinormalsubgroups}, there exists $\chi \in \B{\pi}(G)$ lying over $\theta$. As a consequence, $p \mid \chi(1)$, in contradiction with the hypothesis.

Therefore, we can assume $p \notin \pi$ and, in particular, $p$ does not divide $\abs{N}$. Since $N$ is arbitrarily chosen, we can assume that $\rad{p}{G}=1$. As in the previous paragraph, let $K/N$ be a normal Sylow $p$-subgroup of $G/N$, which is nontrivial because $p$ divides $\abs{G:N}$, and let $C/N=\ze{K/N}$. Note that $N < C \unlhd G$ and $C/N$ is an abelian $\pi'$-group. Then, by Lemma~\ref{lemma:vanishonA}, there exists a character $\chi$ in $\B{\pi}(G)$ such that $\abs{C:N}$ divides $\chi(1)$. However, since $\abs{C:N}$ is a power of $p$, this would contradict the hypothesis.

Finally, if $P$ is a normal Sylow $p$-subgroup of $G$ and $\gamma \in \irr(P)$, then by Theorem~\ref{theorem:Bpinormalsubgroups} there exists $\chi \in \B{\pi}(G) \cup \B{\pi'}(G)$ lying over $\gamma$ and, thus, $\gamma(1) \mid \chi(1)$. Since $p \nmid \chi(1)$, then $\gamma$ is linear. It follows that $\irr(P)=\lin(P)$ and, thus, $P$ is abelian.
\end{proof}

\section{Variants on Thompson theorem for $\B{\pi}$-character}

In this section, we prove Results~\ref{result:thompson1} and \ref{result:thompsonpipiprime}, concerning some variations of Thompson's theorem for $\B{\pi}$-characters and for more then one prime.

For the section, we need a variant of the McKay conjecture, due to T. Wolf. 

\begin{theorem}[\emph{\cite[Theorem 1.15]{W}}]
Let $\pi$ and $\omega$ be two sets of primes and let $G$ be both $\pi$-separable and $\omega$-separable. Let $H$ be a Hall $\omega$-subgroup of $G$ and let $N=\no{G}{H}$. Then:
$$ \abs{\{\chi \in \B{\pi}(G) \mid \chi(1) \mbox{ is a $\omega'$-number}\}} = \abs{\{\psi \in \B{\pi}(N) \mid \psi(1) \mbox{ is a $\omega'$-number}\}}. $$
\end{theorem}

In particular, we need its obvious corollary.

\begin{corollary}
\label{corollary:wolf}
Let $G$ be a $\pi$-separable group and let $H$ be a Hall $\pi$-subgroup of $G$ and let $N=\no{G}{H}$. Then:
$$ \abs{\{\chi \in \B{\pi}(G) \mid \chi(1) \mbox{ is a $\pi'$-number}\}} = \abs{\{\psi \in \B{\pi}(N) \mid \psi(1) \mbox{ is a $\pi'$-number}\}}, $$
$$ \abs{\X{\pi'}(G)}=\abs{\X{\pi'}(N)}=\abs{\irr(N/H)}. $$
\end{corollary}

At first, an easy lemma is needed, which uses the properties of the Fong characters associated with a $\B{\pi}$-character.

\begin{lemma}
\label{lemma:extensionoflinears}
Let $G$ be a $\pi$-separable group and let $H$ be a Hall $\pi$-subgroup, then $\irr_{\pi'}(G) \cap \B{\pi}(G) \subseteq \lin(G)$ if and only if every linear character in $H$ extends to $G$.
\end{lemma}

\begin{proof}
Let $\lambda$ be a linear character in $H$. By Theorem~\ref{theorem:primitivefongchar} and Theorem~\ref{theorem:fongchar}, $\lambda$ is the Fong character associated with some character $\chi \in \irr_{\pi'}(G) \cap \B{\pi}(G)$. It follows that, if $\chi$ is linear, then it extends $\lambda$, while on the other hand if $\lambda$ extends to $G$, then by \cite[Theorem F]{Is3} it has a linear $\pi$-special extension, which coincides with $\chi$ by Theorem~\ref{theorem:fongchar}.
\end{proof}

Now, we can already prove Result~\ref{result:thompson1}, which relates the families of characters $\irr(G)$ and $\B{\pi}(G) \cup \B{\pi'}(G)$ for what concerns the hypothesis of Thompson's theorem.

\begin{proposition}[\emph{Theorem~\ref{result:thompson1}}]
\label{proposition:1}
Let $G$ be a $\pi$-separable group. Then, $\irr_{\pi'}(G)=\lin(G)$ if and only if $\irr_{\pi'}(G) \cap \big(\B{\pi}(G) \cup \B{\pi'}(G) \big) \subseteq \lin(G)$.
\end{proposition}

\begin{proof}
One direction is obviously true. Thus, let one assume $\irr_{\pi'}(G) \cap \B{\pi}(G) \subseteq \lin(G)$ and suppose there exists a nonlinear character $\chi \in \irr(G)$ such that $\chi(1)$ is a $\pi'$-number. By Theorem~\ref{theorem:maximalnucleus}, there exists $W \leq G$, $\alpha \in \X{\pi}(W)$ linear and $\beta \in \X{\pi'}(W)$ such that $\chi=(\alpha\beta)^G$, $W$ contains a Hall $\pi$-subgroup $H$ of $G$ and it is the maximal subgroup of $G$ such that $\alpha_H$ extends to $W$. However, by Lemma~\ref{lemma:extensionoflinears}, $\alpha_H$ extends to $G$, thus $W=G$. It follows that $\beta$ is a nonlinear $\pi'$-special character of $G$, negating the fact that every character in $\X{\pi'}(G) = \irr_{\pi'}(G) \cap \B{\pi'}(G)$ is linear, in contradiction with the hypothesis.
\end{proof}

At this point, we can already prove a related result, concerning a sub-case of Theorem~\ref{result:thompsonpipiprime} and of \cite[Corollary 3]{NW}.

\begin{corollary}
Let $G$ be a $\pi$-separable group and let $H$ be a Hall $\pi$-subgroup, then
\begin{itemize}
\item[i)] $\irr_{\pi'}(G)=\{1_G\}$ if and only if $\irr_{\pi'}(G) \cap \big(\B{\pi}(G) \cup \B{\pi'}(G) \big) = \{1_G\}$;
\item[ii)] $\irr_{\pi'}(G) \cap \B{\pi}(G) = \{1_G\}$ if and only if $H=H'$;
\item[iii)] $\X{\pi'}(G)=\{1_G\}$ if and only if $H$ is self-normalizing;
\item[iv)] $\irr_{\pi'}(G)=\{1_G\}$ if and only if $H=H'$ and $H$ is self-normalizing.
\end{itemize}
\end{corollary}

\begin{proof}
For point (i), only one direction is needed. Suppose, thus, that $\irr_{\pi'}(G) \cap \big(\B{\pi}(G) \cup \B{\pi'}(G) \big) = \{1_G\} \subseteq \lin(G)$, then by Proposition~\ref{proposition:1} $\irr_{\pi'}(G)=\lin(G)$. It follows that every character in $\irr_{\pi'}(G)$ can be factorized as a product $\alpha\beta$, with $\alpha \in \irr_{\pi'} \cap \X{\pi}(G)$ and $\beta \in \X{\pi'}(G)$; however, the two sets of characters both coincide with $\{1_G\}$ by hypothesis.

Point (ii) follows directly from Lemma~\ref{lemma:extensionoflinears}. In fact, if $\irr_{\pi'}(G) \cap \B{\pi}(G) = \{1_G\} \subseteq \lin(G)$, then every character in $\lin(H)$ extends to $G$ and, by \cite[Theorem F]{Is3}, it has an extension in $\irr_{\pi'}(G) \cap \B{\pi}(G)$, thus $\lin(H)=\{1_H\}$. On the other hand, if $\lin(H)=\{1_H\}$, then there are no nonprincipal Fong characters of $H$ in $G$ and it follows that $\abs{\irr_{\pi'}(G) \cap \B{\pi}(G)} = 1$ and the thesis follows.

Finally, point (iii) is a direct consequence of Corollary~\ref{corollary:wolf} and point (iv) follows from points (i), (ii) and (iii).
\end{proof}

Let us now proceed by proving Theorem~\ref{result:thompsonpipiprime}.

\begin{proposition}[\emph{Theorem~\ref{result:thompsonpipiprime}, a)}]
\label{proposition:2}
Let $G$ be a $\pi$-separable group and let $H$ be a Hall $\pi$-subgroup. Then, $\irr_{\pi'}(G) \cap \B{\pi}(G) \subseteq \lin(G)$ if and only if $G' \cap H = H'$.
\end{proposition}

\begin{proof}
From Lemma~\ref{lemma:extensionoflinears} we know that the property that every character in $\irr_{\pi'}(G) \cap \B{\pi}(G)$ is linear is equivalent to the fact that every character in $H/H'$ extends to $G$. Thus, suppose that for every $\lambda \in \irr(H/H')$ there exists $\chi \in \lin(G)$ such that $\chi_H=\lambda$. It follows that
$$ H' \leq G' \cap H \leq \bigcap_{\chi \in \lin(G)} \op{ker}(\chi_H) = \bigcap_{\lambda \in \lin(H)} \op{ker}(\lambda)=H' $$
and, therefore, $G' \cap H = H'$.

On the other hand, suppose $G' \cap H = H'$, then one can write
$$ \frac{G}{G'} = \frac{HG'}{G'} \times \frac{KG'}{G'} = \frac{H}{H'} \times \frac{KG'}{G'} $$
and, thus, every $\lambda \in \irr(H/H')$ extends to $\lambda \times 1_{KG'/G'} \in \irr(G/G')$.
\end{proof}

\begin{proposition}[\emph{Theorem~\ref{result:thompsonpipiprime}, b)}]
\label{proposition:3}
Let $G$ be a $\pi$-separable group, let $H$ be a Hall $\pi$-subgroup for $G$ and let $N=\no{G}{H}$. Then, $\X{\pi'}(G) \subseteq \lin(G)$ if and only if $G' \cap N \leq H$.
\end{proposition}

\begin{proof}
Assume at first that $\X{\pi'}(G) \subseteq \lin(G)$, therefore if $\chi \in \X{\pi'}(G)$, then $\chi_N$ is linear. Suppose that, for some $\chi,\psi \in \X{\pi'}(G)$, $\chi_N=\psi_N$; then $N \leq \op{ker}(\chi\bar{\psi}) \lhd G$. It follows that $\op{ker}(\chi\bar{\psi}) = G$, by Frattini argument, and thus $\chi = \psi$. Therefore, the restriction realizes an injection from $\X{\pi'}(G)$ to $\X{\pi'}(N)$ and, since $\abs{\X{\pi}(G)}=\abs{\X{\pi}(N)}$, by Corollary~\ref{corollary:wolf}, it is actually a bijection. It follows that every character in $\irr(N/H)$ is the restriction of a linear character of $G$, thus we have that
$$ G' \cap N \leq \bigcap_{\chi \in \lin(G)} \op{ker}(\chi_N) \leq \bigcap_{\lambda \in \irr(N/H)} \op{ker}(\lambda) = H. $$

On the other hand, suppose that $G' \cap N \leq H$. Let $X$ be a complement for $H$ in $N$ and note that $X$ is abelian. Moreover, note that $NG'$ is normal in $G$ and it contains $N$, thus $G=NG'$ for the Frattini's argument. It follows that
$$ \frac{G}{G'} \cong \frac{N}{G' \cap N} = X \times \frac{H}{G' \cap H} $$
and, thus, there is a bijection between characters in $\irr(X)=\irr(N/H)$ and characters in $\X{\pi'}(G/G')$. However, by Corollary~\ref{corollary:wolf} we have that $\abs{\X{\pi'}(G/G')} = \abs{\irr(N/H)} = \abs{\X{\pi'}(G)}$ and, thus, it follows that every $\pi'$-special character in $G$ is linear.
\end{proof}

As anticipated in the introduction, it can be easily seen that \cite[Corollary 3]{NW} can also be obtained as a corollary of these last results.

\begin{corollary}[\emph{\cite[Corollary 3]{NW}}]
Let $G$ be a $\pi$-separable group and let $H$ be a Hall $\pi$-subgroup for $G$ and $N=\no{G}{H}$. Then, $\irr_{\pi'}(G)=\lin(G)$ if and only if $G' \cap N = H'$.
\end{corollary}

\begin{proof}
If $N$ is the normalizer in $G$ of a Hall $\pi$-subgroup $H$, Propositions \ref{proposition:1}, \ref{proposition:2} and \ref{proposition:3} provide that $\irr_{\pi'}(G)=\lin(G)$ if and only if both $G' \cap N \leq H$ and $G' \cap H = H'$ and the thesis follows directly from this.
\end{proof}


\begin{thebibliography}{99}

	\bibitem{DPSS}
	\begin{sc}S. Dolfi, E. Pacifici, L. Sanus, P. Spiga\end{sc}:
	``On the orders of zeros of irreducible characters''.
	\emph{Journal of Algebra} 321 (2009), 345-352.

	\bibitem{G}
	\begin{sc}D. Gajendragadkar\end{sc}:
	``A characteristic class of characters finite of $\pi$-separable groups''.
	\emph{Journal of Algebra} 59 (1979), 237-259.
	
	\bibitem{GR}
	\begin{sc}K. W. Gruenberg, K. W. Roggenkamp\end{sc}:
	``Decomposition of the augmentation ideal and of the relation modules of a finite group''.
	\emph{Proc. London Math. Soc.} 31 (1975), 149--166.
	
	\bibitem{Is}
	\begin{sc}I. M. Isaacs\end{sc}:
	\emph{Character Theory of Finite Groups}.
	Academic Press, (1976).

	\bibitem{Is1}
	\begin{sc}I. M. Isaacs\end{sc}:
	``Characters of $\pi$-separable groups''.
	\emph{Journal of Algebra} 86 (1984) 98--128.
	
	\bibitem{Is2}
	\begin{sc}I. M. Isaacs\end{sc}:
	``Fong characters of $\pi$-separable groups''.
	\emph{Journal of Algebra} 99 (1986) 89--107.
	
	\bibitem{Is3}
	\begin{sc}I. M. Isaacs\end{sc}:
	``Induction and restriction of $\pi$-special characters''.
	\emph{Can. J. Math.}, Vol. XXXVIII, No. 3 (1986) 576--604.
	
	\bibitem{Is4}
	\begin{sc}I. M. Isaacs\end{sc}:
	\emph{Characters of solvable groups}.
	Graduate Studies in Mathematics, 189 (2018).	
	
	\bibitem{IN}
	\begin{sc}M. I. Isaacs, G. Navarro\end{sc}:
	``Characters of $p'$-degree of $p$-solvable groups''.
	\emph{Journal of Algebra} 246 (2001) 394--413.
	
	\bibitem{Na}
	\begin{sc}G. Navarro\end{sc}:
	``Zeros of Primitive Characters in Solvable Groups''.
	\emph{Journal of Algebra} 221 (1999), 644--650.	
	
	\bibitem{NW}
	\begin{sc}G. Navarro, T. R. Wolf\end{sc}:
	``Variations on Thompson's character degrees theorem''.
	\emph{Glasgow Math. J.} 44 (2002) 371--374.
	
	\bibitem{W}
	\begin{sc}T. R. Wolf\end{sc}:
	``Variations on McKay's character degrees conjecture''.
	\emph{Journal of Algebra} 135 (1990) 123--138.
	
	
	
	
	
\end{thebibliography}
\end{document}